\documentclass[12pt]{article}

\usepackage{amsmath, amssymb,amsthm}

\usepackage{color}
\definecolor{darkred}{RGB}{139,0,0}
\definecolor{darkgreen}{RGB}{0,100,0}
\definecolor{darkmagenta}{RGB}{139,0,139}
\definecolor{darkpurple}{RGB}{110,0,180}
\definecolor{darkblue}{RGB}{40,0,200}
\definecolor{darkorange}{RGB}{255,140,0}

\newcommand{\bsx}{\boldsymbol{x}}

\newcommand{\bszero}{\boldsymbol{0}}

\newcommand{\bst}{\boldsymbol{t}}
\newcommand{\rd}{\,{\rm d}}

\newcommand{\RR}{\mathbb{R}}
\newcommand{\NN}{\mathbb{N}}
\newcommand{\EE}{\mathbb{E}}
\newcommand{\PP}{\mathbb{P}}
\newcommand{\cP}{\mathcal{P}}

\newtheorem{definition}{Definition}
\newtheorem{theorem}{Theorem}
\newtheorem{remark}{Remark}
\newtheorem{lemma}{Lemma}

\title{Tractability properties of the discrepancy in\\ Orlicz norms}

\author{Josef Dick\thanks{Josef Dick is partly supported by the Australian Research Council Discovery Project DP190101197.}, Aicke Hinrichs\thanks{A. Hinrichs is
 supported by the Austrian Science Fund (FWF): Project F5513-N26,
which is a part of the Special Research Program ``Quasi-Monte Carlo Methods:
Theory and Applications''.}, Friedrich Pillichshammer\thanks{F. Pillichshammer is
 supported by the Austrian Science Fund (FWF): Project F5509-N26,
which is a part of the Special Research Program ``Quasi-Monte Carlo Methods:
Theory and Applications''.}\ \footnote{Corresponding author}, and Joscha Prochno\thanks{J. Prochno is
 supported by the Austrian Science Fund (FWF): Project P32405 ``Asymptotic Geometric Analysis and Applications'' and a visiting professorship from the University of Bochum and its Research School PLUS.}}

 \date{\today}

\begin{document}

\maketitle

\centerline{\large Dedicated to  Gerhard Larcher on the occasion of his $60^{{\rm th}}$ birthday\footnote{Heinrich {\it et al.} mention in their seminal paper \cite[p.~280]{HNWW} that the question about the  dependence on $d$ for the star-discrepancy was first raised by Gerhard Larcher in 1998.}.}

\begin{abstract}
We show that the minimal discrepancy of a point set in the $d$-dimensional unit cube with respect to Orlicz norms can exhibit both polynomial and weak tractability. In particular, we show that the $\psi_\alpha$-norms of exponential Orlicz spaces are polynomially tractable.
\end{abstract}

\centerline{\begin{minipage}[hc]{130mm}{
}
\end{minipage}}

\section{Introduction and main results}

The discrepancy of an $N$-element point set $\cP = \{\bsx_1, \bsx_2, \ldots, \bsx_N\}$ in the unit cube $[0,1]^d$ measures the deviation of the empirical distribution of $\cP$ from the uniform measure. This concept has important applications in numerical analysis, where so-called Koksma-Hlawka inequalities establish a deep connection between norms of the discrepancy function and worst case errors of quasi-Monte Carlo integration rules determined by the point set $\cP$. For a comprehensive introduction and exposition on this subject we refer the reader to \cite{DP10,H2013,KN1974} and the references cited therein. 

To define the concept of discrepancy, we first introduce the local discrepancy function $\Delta_{\cP}: [0,1]^d \to \mathbb{R}$ defined as 
\begin{equation*}
\Delta_{\cP}(\bst) = \frac{\#\{ j \in \{1, 2, \ldots, N\}: \bsx_j \in [\bszero, \bst) \} }{N} - \mathrm{Vol}([\bszero, \bst]),
\end{equation*}
where $[\bszero, \bst) = [0, t_1) \times [0, t_2) \times \ldots \times [0, t_d)$ for $\bst = (t_1, t_2, \ldots, t_d) \in [0,1]^d$ and $\mathrm{Vol}(\cdot)$ stands for the $d$-dimensional Lebesgue measure. We now apply a norm $\| \cdot \|_{\bullet}$ to the local discrepancy function to obtain the discrepancy $\| \Delta_{\cP}\|_{\bullet}$ of the point set $\cP$ with respect to the norm $\|\cdot \|_{\bullet}$. Of particular interest are the norms on the usual Lebesgue spaces $L_p$ ($1 \le p \leq \infty$) of $p$-integrable functions on the unit cube $[0,1]^d$. Those lead to the central notions of $L_p$-discrepancy for $1\leq p<\infty$, and the $L_\infty$-discrepancy, which is usually called the star-discrepancy, when $p=\infty$.

The $N^{\mathrm{th}}$ minimal discrepancy with respect to the norm $\|\cdot \|_{\bullet}$ in dimension $d$ is the best possible discrepancy over all point sets of size $N$ in the $d$-dimensional unit cube $[0,1]^d$, i.e.,
\begin{equation*}
\mathrm{disc}_{\bullet}(N,d) = \inf_{\stackrel{\cP \subseteq [0,1]^d}{|\cP| = N}} \| \Delta_{\cP}\|_{\bullet}.
\end{equation*}
We compare this value with the initial discrepancy given by the discrepancy of the empty point set $\| \Delta_{\emptyset} \|_{\bullet}$. Since the initial discrepancy may depend on the dimension, we use it to normalize the $N^{\mathrm{th}}$ minimal discrepancy when we study the dependence of $\mathrm{disc}_{\bullet}(N,d)$ on the dimension $d$. We therefore define the inverse of the $N^{\mathrm{th}}$ minimal discrepancy in dimension $d$ as the number $N_{\bullet}(\varepsilon, d)$ which is the smallest number $N$ such that a point set with $N$ points exists that reduces the initial discrepancy at least by a factor of $\varepsilon\in(0,1)$,
\begin{equation*}
N_{\bullet}(\varepsilon, d) = \min \big\{ N \in \mathbb{N}\,:\, \mathrm{disc}_{\bullet}(N, d) \le \varepsilon \|\Delta_\emptyset \|_{\bullet} \big\}.
\end{equation*}
In this paper we are interested in how $N_{\bullet}(\varepsilon, d)$ depends simultaneously on $\varepsilon$ and the dimension $d$. In general, the dependence of the inverse of the $N^{\mathrm{th}}$ minimal discrepancy can take different forms. For instance, if the dependence on the dimension $d$ or on $\varepsilon^{-1}$ is exponential, then we call the discrepancy intractable. If the inverse of the $N^{\mathrm{th}}$ minimal discrepancy grows exponentially fast in $d$, then the discrepancy is said to suffer from the curse of dimensionality. On the other hand, if $N_{\bullet}(\varepsilon, d)$ increases at most polynomially in $d$ and $\varepsilon^{-1}$, as $d$ increases and $\varepsilon$ tends to zero, then the discrepancy is said to be polynomially tractable. This leads us to the following definition.

\begin{definition}\label{def_poly_tract}
The discrepancy with respect to the norm $\| \cdot \|_{\bullet}$ is \textit{polynomially tractable} if there are numbers $C \in(0,\infty)$, $\tau \in(0,\infty)$, and $\sigma \in(0,\infty)$ such that
\begin{equation}\label{condPTdef}
N_{\bullet}(\varepsilon, d) \le C\, d^\tau \varepsilon^{-\sigma}, \quad \mbox{for all } \varepsilon \in (0,1) \mbox{ and all } d \in \mathbb{N}.
\end{equation}
The infimum over all exponents $\tau\in(0,\infty)$ such that a bound of the form \eqref{condPTdef} holds is called the $d$-exponent of polynomial tractability.
\end{definition}

To cover cases between polynomial tractability and intractability, we now introduce the concept of weak tractability, where $N_{\bullet}(\varepsilon, d)$ is not exponential in $\varepsilon^{-1}$ and $d$. This encodes the absence of intractability.

\begin{definition}\label{def_weakly_tract}
The discrepancy with respect to the norm $\| \cdot \|_{\bullet}$ is weakly tractable, if 
\begin{equation*}
\lim_{d+\varepsilon^{-1}  \to \infty} \frac{\log N_{\bullet}(\varepsilon, d)}{d+\varepsilon^{-1}} = 0.
\end{equation*}
(Throughout this paper $\log$ means the natural logarithm.)
\end{definition}

The subject of tractability of multivariate problems is a very popular and active area of research and we refer the reader to the books \cite{NW08,NW10} by Novak and Wo\'{z}niakowski for an introduction into tractability studies of discrepancy and an exhaustive exposition.

A famous result by Heinrich, Novak, Wasilkowski, and Wo\'{z}niakowski \cite{HNWW} based on the theory of empirical processes and Talagrand's majorizing measure theorem shows that the star-discrepancy is polynomially tractable. In fact, they show that $\tau$ in Definition~\ref{def_poly_tract} can be set to one and hence in this case the inverse of the star-discrepancy $N_{L_\infty}(\varepsilon, d)$ depends at most linearly on the dimension $d$. 
It was shown in \cite{HNWW} and \cite{H2004} that $\tau=1$ is the minimal possible $\tau$ in Definition~\ref{def_poly_tract} for the star-discrepancy. Determining the optimal exponent $\sigma$ for $\varepsilon^{-1}$ is an open problem.
On the other hand, the $L_2$-discrepancy is known to be intractable, as shown by Wo\'{z}niakowski \cite{W99} (see also \cite{NW10}). The behavior of the inverse of the $L_p$-discrepancy in between, where $p \notin \{2, \infty\}$, seems to be unknown.

Note that due to the normalization with the initial discrepancy, we cannot infer a continuous change in the behavior of $N_{L_p}(\varepsilon,d)$ as $p$ goes from $1$ to $\infty$. A natural assumption seems to be that the $L_p$-discrepancy is intractable for any $p \in [1, \infty)$. If correct, this would mean that there is a sharp change from intractability to polynomial tractability as one goes from $p\in [1, \infty)$ to $p=\infty$. A natural question which hence arises is what happens between those two cases $p \in [1, \infty)$ and $p = \infty$.

To study this question we shall work in the setting of (specific) Orlicz spaces. Let us recall that a function $M:[0,\infty)\to[0,\infty)$ is said to be an Orlicz function if $M(0)=0$, $M$ is convex, and $M(t)>0$ for $t>0$. If $\lim_{x\rightarrow 0} M(x)/x=\lim_{x\rightarrow \infty}x/M(x)=0$, then $M$ is called an $N$-function. The previous limit assumptions simply guarantee that the convex-dual is again an $N$-function. Now if $D\subseteq \RR^d$ is a compact set, we define the Orlicz space $L_M$ to be the space of (equivalence classes of) Lebesgue measurable functions $f$ on $D$ for which
\[
\|f\|_{M} := \inf \left \{ K > 0: \int_{D} M \left( \frac{|f(\bsx)|}{K} \right) \, \mathrm{d} \bsx \le 1 \right\} < \infty.
\]
The latter functional is a norm on $L_M$ known as Luxemburg norm, named after W.~A.~J.~Luxemburg~\cite{Lux}, which turns $L_M$ into a Banach space. One commonly just speaks of Orlicz functions, Orlicz norms, and Orlicz spaces. An introduction to the theory of Orlicz spaces can be found in \cite{KR1961}.

For our purpose, we introduce for $\alpha\in[1,\infty)$ the exponential Orlicz norms $\|\cdot \|_{\psi_\alpha}$, which for a measurable function $f$ defined on $[0,1]^d$ are given by
\begin{equation*}
\| f \|_{\psi_\alpha} = \inf \left \{ K > 0: \int_{[0,1]^d} \psi_\alpha \left( \frac{|f(\bsx)|}{K} \right) \, \mathrm{d} \bsx \le 1 \right\},
\end{equation*}
where $\psi_\alpha(x) = \exp(x^\alpha) - 1$. The assumption $\alpha \geq 1$ guarantees the convexity of $\psi_\alpha$. These norms play an important role in the study of the concentration of mass in high-dimensional convex bodies \cite{AGM_book, BGVV2014, LT77} and have recently found applications in the tractability study of multivariate numerical integration \cite{HPU2019}. They have appeared earlier in discrepancy theory and the related multivariate integration problems in fixed dimension \cite{BLPV2009, BM2018, DHMP2017}. As we shall see later, the discrepancy with respect to $\psi_\alpha$-norms turns out to be polynomially tractable as well.

In our context it is interesting to also study variations of these norms exhibiting different types of behavior of $N_{\bullet}(\varepsilon, d)$ as a function of the dimension $d$. In fact, we may write $\psi_\alpha$ as the series
\begin{equation*}
\psi_\alpha(x) = \frac{x^\alpha}{1!} + \frac{x^{2 \alpha}}{2!} + \frac{x^{3 \alpha}}{3!} + \cdots
\end{equation*}
and consider the more general case where $\psi_{\alpha}$ is replaced by a function
\begin{equation}\label{def:psi}
\psi_{\alpha, \varphi}(x) = \frac{x^\alpha}{ ( \varphi(\alpha) )^{\alpha} } + \frac{x^{2\alpha}}{ (\varphi(2 \alpha ))^{2 \alpha} } + \frac{x^{3\alpha}}{( \varphi(3 \alpha ))^{3 \alpha} } + \cdots
\end{equation}
for a non-decreasing function $\varphi: [0,\infty) \to (0,\infty)$ with $\lim_{x \rightarrow \infty}\varphi(x)=\infty$. Note that the growth condition on $\varphi$ guarantees, according to the ratio test, the absolute convergence of the series \eqref{def:psi} for all $x \in [0,\infty)$. Choosing $\varphi(p \alpha) = ( p! )^{1/(p \alpha)}$ takes us back to the $\psi_\alpha$-norm, which is therefore a special case of the more general setting.

Below we will characterize functions $\varphi$ for which the discrepancy with respect to $\| \cdot \|_{\psi_{\alpha, \varphi}}$, given by
\begin{equation*}
\| f \|_{\psi_{\alpha, \varphi}} = \inf \left \{ K > 0: \int_{[0,1]^d} \psi_{\alpha, \varphi} \left( \frac{|f(\bsx)|}{K} \right) \, \mathrm{d} \bsx \le 1 \right\},
\end{equation*}
is polynomially tractable and weakly tractable. 

The aim of this paper is to show the following result.

\begin{theorem}\label{thm1}
Let $\alpha\in[1,\infty)$. Then the following hold:
\begin{enumerate}
\item The discrepancy with respect to the $\psi_\alpha$-norm $\| \cdot \|_{\psi_\alpha}$ is polynomially tractable.
\item For any non-decreasing $\varphi:[0,\infty)\to (0,\infty)$ with $\lim_{x \rightarrow \infty}\varphi(x)=\infty$ for which there exists an $r \ge 0$ and a constant $C \in(0,\infty)$ such that for all $p\geq 1$
\begin{equation}\label{condPT}
\varphi(p) \le C\, p^r,
\end{equation}
the discrepancy with respect to $\| \cdot \|_{\psi_{\alpha, \varphi}}$ is polynomially tractable. The $d$-exponent of polynomial tractability is at most $3+2r$.
\item For any non-decreasing $\varphi:[0,\infty)\to (0,\infty)$ with $\lim_{x \rightarrow \infty}\varphi(x)=\infty$ which satisfies
\begin{equation}\label{condWT}
\lim_{p \rightarrow \infty} \frac{\log \varphi(p)}{p} =0,
\end{equation}
the discrepancy with respect to $\| \cdot \|_{\psi_{\alpha, \varphi}}$ is weakly tractable.
\end{enumerate}
\end{theorem}

\begin{remark}\rm
Note that by choosing $\psi_{\alpha, \varphi}(p) = p^\alpha$ we obtain the classical $L_\alpha$-norm. In this case $\varphi(\alpha)=1$ and $\varphi(x)=\infty$ for all $x>\alpha$. This choice of $\varphi$ does not satisfy any of the conditions in Theorem~\ref{thm1}.

An example of a function $\varphi$ that satisfies condition \eqref{condWT} for weak tractability is $\varphi(p)=\exp(p^{\tau})$ with some $\tau \in (0,1)$. This function does not satisfy condition \eqref{condPT}.
\end{remark}
 
We can in fact provide a more accurate estimate for the exponential Orlicz norms and the $d$-exponent of polynomial tractability.

\begin{theorem}\label{thm2}
For any $\alpha \in[1,\infty)$, we have
\begin{equation*}
N_{\psi_{\alpha}}(\varepsilon, d) \le \left\lceil C_{\alpha} \, d^{\max\{1,2/\alpha\}}\, (\log(d+1))^{2/\alpha}    \varepsilon^{-2}\right\rceil,
\end{equation*}
where $$C_{\alpha}= 2601 \cdot \alpha^{2/\alpha} \cdot \left(\frac{\sqrt{2 \pi}}{{\rm e}^{11/12}}\right)^{2/\alpha} .$$
In particular, the $d$-exponent of polynomial tractability is at most $\max\{1,2/\alpha\}$.  
\end{theorem}

This upper bound on $N_{\psi_\alpha}(\varepsilon,d)$ shows that for $\alpha \to \infty$ the inverse of the star-discrepancy depends linearly on the dimension, thereby matching the result of Heinrich, Novak, Wasilkowski, and Wo\'{z}niakowski \cite{HNWW}.
 
\vskip 2mm
In the following Section~\ref{sec:proofs} we present the proofs of our main results, where we start by establishing an equivalence between the norms $\| \cdot\|_{\psi_{\alpha,\varphi}}$ and an expression involving a supremum of classical $L_p$-norms. Subsection~\ref{subsec:thm1} is then devoted to the proof of Theorem~\ref{thm1}. The proof of Theorem~\ref{thm2} will be presented in Subsection~\ref{subsec:thm2}.

\section{The proofs}\label{sec:proofs}

For the proofs of Theorems~\ref{thm1} and \ref{thm2} we define another norm which we show to be equivalent to the Orlicz norm $\| \cdot \|_{\psi_{\alpha, \varphi}}$, namely
\begin{equation}\label{eq:phi norm}
\| f\|_{\varphi} := \sup_{p \ge 1} \frac{\|f\|_{L_p}}{\varphi(p)}
\end{equation}
with $\varphi:[0,\infty)\to (0,\infty)$. In the special case of exponential Orlicz norms $\| \cdot \|_{\psi_\alpha}$ such an equivalence is a classical result in asymptotic geometric analysis and may be found, without explicit constants, in the monographs \cite[Lemma 3.5.5]{AGM_book} and \cite[Lemma 2.4.2]{BGVV2014}. In the context of this paper it is important that these constants do not depend on the dimension $d$.

\begin{lemma}\label{lem1}
Let $d\in\NN$ and $\alpha\in[1,\infty)$. For any measurable function $f: [0,1]^d \to \mathbb{R}$, we have the estimates
\begin{equation}\label{le:ne1}
\inf_{p \ge 1} \frac{\varphi(p)}{\max\{\varphi(\alpha), \varphi(p)\}}  \ \|f \|_{\varphi}  \le \| f\|_{\psi_{\alpha, \varphi}} \le 2^{1/\alpha} 
\| f \|_{\varphi}.
\end{equation}
In particular, for any $\alpha \in[1,\infty)$, we have
\begin{equation}\label{le:ne2}
\left( \frac{\mathrm{e}^{11/12}}{ \sqrt{2 \pi}}  \right)^{1/\alpha} 
\| f\|_\alpha \le \|f \|_{\psi_\alpha} \le (2\, \mathrm{e}\, \alpha )^{1/\alpha} 
\,\|f \|_\alpha, 
\end{equation}
where $\| f\|_\alpha:= \sup_{p \ge 1} p^{-1/\alpha} \|f\|_{L_p}$.
\end{lemma}

\begin{proof}
Using the series expansion of $\psi_{\alpha, \varphi}$, we obtain
\begin{align*}
\int_{[0,1]^d} \psi_{\alpha, \varphi} \left(\frac{|f(\bsx)|}{K} \right) \,\mathrm{d} \bsx = & \sum_{p=1}^\infty  \left( \frac{\|f \|_{L_{\alpha p}}}{K  \varphi( \alpha p)  } \right)^{\alpha p}.
\end{align*}
By choosing
\begin{equation*}
K= 2^{1/\alpha} \sup_{p \ge \alpha} \frac{ \| f \|_{L_{p}}}{ \varphi(p) },
\end{equation*}
we obtain
\begin{equation*}
\int_{[0,1]^d} \psi_{\alpha, \varphi} \left(\frac{|f(\bsx)|}{K} \right) \,\mathrm{d} \bsx \le  \sum_{p=1}^\infty 2^{-p}  = 1.
\end{equation*}
Therefore, we have 
$$
\| f\|_{\psi_{\alpha, \varphi}} \le K= 2^{1/\alpha} \sup_{p \ge \alpha} \frac{ \| f \|_{L_{p}}}{ \varphi(p) }\,.
$$
This implies the upper bound in \eqref{le:ne1} for all $\alpha \ge 1$. 

For the lower bound, we argue as follows. For any $K \in(0,\infty)$ such that $K\geq \| f\|_{\psi_{\alpha, \varphi}}$, we have
\begin{equation*}
1 \ge \int_{[0,1]^d} \psi_{\alpha, \varphi} \left( \frac{|f (\bsx)|}{K} \right) \,\mathrm{d} \bsx =  \sum_{p=1}^\infty  \left( \frac{\|f \|_{L_{\alpha p}}}{K  \varphi( \alpha p)  } \right)^{\alpha p}  \ge \left( \sup_{p \ge 1} \frac{\| f\|_{L_{\alpha p}}}{K \varphi(\alpha p)}  \right)^{\alpha p}.
\end{equation*}
This implies that
\begin{equation*}
 K \ge \sup_{p \ge \alpha} \frac{\|f \|_{L_p}}{ \varphi(p)}
\end{equation*}
and since this holds for any such $K$, we obtain
\begin{equation*}
\| f\|_{\psi_{\alpha, \varphi}} \ge \sup_{p \ge \alpha} \frac{\|f \|_{L_p}}{ \varphi(p)}.
\end{equation*}

If $\alpha \in[1,\infty)$, and $q \in [1,\alpha]$, then
$$\sup_{p \ge \alpha} \frac{\|f \|_{L_p}}{ \varphi(p)} \ge  \frac{\|f \|_{L_\alpha}}{ \varphi(\alpha)} \ge  \frac{\|f \|_{L_q}}{ \varphi(q)} \inf_{q \in [1,\alpha]}\frac{\varphi(q)}{\varphi(\alpha)}. $$
Hence,  
$$\sup_{p \ge \alpha} \frac{\|f \|_{L_p}}{ \varphi(p)} \ge  \min\left\{1,\inf_{q \in [1,\alpha]}\frac{\varphi(q)}{\varphi(\alpha)}\right\} \sup_{p \ge 1} \frac{\|f \|_{L_p}}{ \varphi(p)}.$$

In any case, for all $\alpha\in[1,\infty)$, we have that
$$
\sup_{p \ge \alpha} \frac{\|f \|_{L_p}}{ \varphi(p)} \ge  \min\left\{1,\inf_{q \ge 1}\frac{\varphi(q)}{\varphi(\alpha)}\right\}\, \|f\|_{\varphi} =  \inf_{q \ge 1} \frac{ \varphi(q)}{ \max\{ \varphi(\alpha), \varphi(q)\}}\ \|f\|_{\varphi} ,
$$
which implies the result since $\inf_{q \ge 1} \frac{ \varphi(q)}{ \max\{\varphi(\alpha), \varphi(q)\}} \le 1$.\\

The bound \eqref{le:ne2} for the $\psi_\alpha$-norms can be shown using similar arguments together with Stirling's formula
\begin{equation}\label{ineq_Stirling}
\sqrt{2 \pi p} (p/ \mathrm{e})^{p} \le p! \le \sqrt{2\pi p} (p/\mathrm{e})^p \mathrm{e}^{1/(12 p)}.
\end{equation}

We use the Taylor series expansion of the exponential function and obtain
\begin{eqnarray}\label{eq1}
\int_{[0,1]^d} \psi_{\alpha}\left(\frac{|f(\bsx)|}{K}\right) \rd \bsx & = &  \int_{[0,1]^d} \sum_{\ell=1}^{\infty} \frac{1}{\ell!} \left(\frac{|f(\bsx)|}{K}\right)^{\alpha \ell} \rd \bsx\nonumber\\
& = & \sum_{\ell=1}^{\infty} \frac{1}{\ell !} \left(\frac{\|f\|_{L_{\alpha \ell}}}{K}\right)^{\alpha \ell}.  
\end{eqnarray}
Using Stirling's formula \eqref{ineq_Stirling} we get 
\begin{eqnarray*}
\int_{[0,1]^d} \psi_{\alpha}\left(\frac{|f(\bsx)|}{K}\right) \rd \bsx \le \sum_{\ell=1}^{\infty} \frac{{\rm e}^{\ell}}{\ell^{\ell}} \left(\frac{\|f\|_{L_{\alpha \ell}}}{K}\right)^{\alpha \ell}= \sum_{\ell=1}^{\infty} \left(\frac{\|f\|_{L_{\alpha \ell}}\, ({\rm e} \alpha)^{1/\alpha} }{K (\ell \alpha)^{1/\alpha}}\right)^{\alpha \ell}.
\end{eqnarray*}
If we choose $$K=(2{\rm e} \alpha)^{1/\alpha} \,\sup_{\ell \ge 1} \frac{\|f\|_{L_{\alpha \ell}}}{(\ell \alpha)^{1/\alpha}},$$ then we obtain $$\int_{[0,1]^d} \psi_{\alpha}\left(\frac{|f(\bsx)|}{K}\right) \rd \bsx \le \sum_{\ell=1}^{\infty} \frac{1}{2^{\ell}}=1.$$ Hence $$\|f\|_{\psi_{\alpha}} \le K =  (2 {\rm e} \alpha)^{1/\alpha} \,\sup_{\ell \ge 1} \frac{\|f\|_{L_{\alpha \ell}}}{(\alpha \ell)^{1/\alpha}} =  (2 {\rm e} \alpha)^{1/\alpha} \,\sup_{p \ge \alpha} \frac{\|f\|_{L_p}}{p^{1/\alpha}} \le (2 {\rm e} \alpha)^{1/\alpha} \,\|f\|_{\alpha}.$$

On the other hand, from \eqref{eq1} and the upper bound in Stirling's formula \eqref{ineq_Stirling} we obtain
\begin{eqnarray*}
\int_{[0,1]^d} \psi_{\alpha}\left(\frac{|f(\bsx)|}{K}\right) \rd \bsx & \ge & \sum_{\ell=1}^{\infty} \frac{1}{\sqrt{2 \pi \ell} \, {\rm e}^{1/(12 \ell)}} \left(\frac{{\rm e}}{\ell}\right)^{\ell} \left(\frac{\|f\|_{L_{\alpha \ell}}}{K}\right)^{\alpha \ell}\\
& \ge & \sup_{\ell \ge 1} \frac{({\rm e} \alpha)^{\ell}}{\sqrt{2 \pi \ell} \, {\rm e}^{1/(12 \ell)}}  \frac{1}{K^{\alpha \ell}} \left(\frac{\|f\|_{L_{\alpha \ell}}}{(\alpha \ell)^{1/\alpha}}\right)^{\alpha \ell}.
\end{eqnarray*}
Now, in order to have $\int_{[0,1]^d} \psi_{\alpha}\left(\frac{|f(\bsx)|}{K}\right) \rd \bsx \le 1$ we find that $K$ has to satisfy $$K^{\alpha \ell} \ge  \frac{({\rm e} \alpha)^{\ell}}{\sqrt{2 \pi \ell} \, {\rm e}^{1/(12 \ell)}}  \left(\frac{\|f\|_{L_{\alpha \ell}}}{(\alpha \ell)^{1/\alpha}}\right)^{\alpha \ell}$$ for all $\ell\ge 1$. Hence $$K \ge  \frac{({\rm e} \alpha)^{1/\alpha}}{(\sqrt{2 \pi \ell} \, {\rm e}^{1/(12 \ell)})^{1/(\alpha \ell)}}  \frac{\|f\|_{L_{\alpha \ell}}}{(\alpha \ell)^{1/\alpha}} \ge  \left(\frac{{\rm e} \alpha}{\sqrt{2 \pi} \, {\rm e}^{1/12}}\right)^{1/\alpha}\,  \frac{\|f\|_{L_{\alpha \ell}}}{(\alpha \ell)^{1/\alpha}}  $$ for all $\ell\ge 1$. Hence, $$K \ge \left(\frac{{\rm e} \alpha}{\sqrt{2 \pi} \, {\rm e}^{1/12}}\right)^{1/\alpha}\, \sup_{\ell \ge 1} \frac{\|f\|_{L_{\alpha \ell}}}{(\alpha \ell)^{1/\alpha}} = \left(\frac{{\rm e} \alpha}{\sqrt{2 \pi} \, {\rm e}^{1/12}}\right)^{1/\alpha}\, \sup_{p \ge \alpha} \frac{\|f\|_{L_p}}{p^{1/\alpha}}.$$ 
For any $q \in [1,\alpha]$ we have $$ \sup_{p \ge \alpha} \frac{\|f\|_{L_p}}{p^{1/\alpha}}\ge \frac{\|f\|_{L_\alpha}}{\alpha^{1/\alpha}}\ge \frac{\|f\|_{L_q}}{q^{1/\alpha}} \frac{q^{1/\alpha}}{\alpha^{1/\alpha}} \ge \frac{1}{\alpha^{1/\alpha}} \frac{\|f\|_{L_q}}{q^{1/\alpha}}.$$ Hence
$$\sup_{p \ge \alpha} \frac{\|f\|_{L_p}}{p^{1/\alpha}} \ge \frac{1}{\alpha^{1/\alpha}}\, \sup_{p \ge 1} \frac{\|f\|_{L_p}}{p^{1/\alpha}}.$$

This implies $$\|f\|_{\psi_{\alpha}} \ge \left(\frac{{\rm e}^{11/12}}{\sqrt{2 \pi}}\right)^{1/\alpha}\, \|f\|_{\alpha}$$ as desired. This closes the proof.
\end{proof}

We are now prepared to present the proofs of our main results.

\subsection{The proof of Theorem~\ref{thm1}} \label{subsec:thm1}

An important consequence of Lemma~\ref{lem1} is that the constants do not depend on the dimension, and hence the Orlicz norm discrepancy satisfies the same tractability properties as the discrepancy with respect to the norm $\| \cdot \|_{\varphi}$. Therefore in the following proof we will only use the latter norm.

It is well known and easily checked (see, e.g., \cite[p.~54]{NW10}) that for every $p \in [1,\infty)$, the initial $L_p$-discrepancy in dimension $d$ satisfies
\begin{equation*}
\| \Delta_\emptyset \|_{L_p} = \frac{1}{(p+1)^{d/p}}.
\end{equation*}
If $p = \infty$, then the initial discrepancy is $1$ for every dimension $d \in \mathbb{N}$. This implies that
\begin{equation*}
\| \Delta_\emptyset \|_{\varphi} = \sup_{p \ge 1} \frac{1}{ \varphi(p)} \frac{1}{(p+1)^{d/p}} \ge \frac{1}{(d+1) \varphi(d)},
\end{equation*}
where we used the choice $p = d$ to obtain the last inequality.

From \cite{HNWW} we know that
\begin{equation}\label{HNWWbound}
\mathrm{disc}_{L_\infty}(N, d) \le C_{{\rm PT}}\, \sqrt{\frac{d}{N}},
\end{equation}
for some absolute constant $C_{{\rm PT}} \in(0,\infty)$. Aistleitner~\cite{Aist} showed that one can choose $C_{{\rm PT}} = 10$, but according to \cite{GH19} the constant $C_{{\rm PT}}$ may be reduced to $C_{{\rm PT}} = 2.5287$. 

Hence, we have
\begin{equation*}
\mathrm{disc}_{\varphi}(N, d)  \le  \mathrm{disc}_{L_\infty}(N,d)\cdot \sup_{p \ge 1} \frac{1}{\varphi(p)} \le C_{{\rm PT}}\, \sup_{p \ge 1} \frac{1}{\varphi(p)}  \,\sqrt{\frac{d}{N}}\,,
\end{equation*}
where $\mathrm{disc}_{\varphi}(N, d)$ stands for the discrepancy with respect to the norm $\|\cdot\|_\varphi$ introduced in \eqref{eq:phi norm}. This implies that
\begin{align}\label{Nbound1}
N_{\varphi}(\varepsilon, d) \le & \min \left \{ N \in \mathbb{N}\,:\,  C_{{\rm PT}}\, \sup_{p \ge 1} \frac{1}{\varphi(p)}\, \sqrt{\frac{d}{N}}  \le  \frac{\varepsilon }{(d+1) \varphi(d)} \right\} \nonumber \\  \le &  \left \lceil C_{{\rm PT}}^2 \frac{d (d+1)^2 \varphi^2(d) }{\varepsilon^{2} } \sup_{p \ge 1} \frac{1}{ \varphi^2(p) } \right\rceil,
\end{align}
where for $x\in\RR$, $\lceil x \rceil:=\min\{ n\in\mathbb Z\,:\, n\geq x\}$.  
This concludes the proof of the second statement in Theorem~\ref{thm1}. As mentioned above, if we choose $\varphi(\alpha p) = (p!)^{1/(\alpha p)}$, then we obtain the $\psi_\alpha$-norm. Using Stirling's formula \eqref{ineq_Stirling} together with the previous result, we can deduce the first part of Theorem~\ref{thm1}.

In order to prove the third part of Theorem~\ref{thm1}, we apply the logarithm to $N_{\varphi}(\varepsilon, d)$. From \eqref{Nbound1} we obtain that 
$$\log N_{\varphi}(\varepsilon, d) \le C' +2 \log \varepsilon^{-1} + 3 \log (d+1) + 2 \log \varphi(d)$$ 
for some $C'\in(0,\infty)$ only depending on $\varphi$. 
Hence,
$$\limsup_{d+\varepsilon^{-1} \rightarrow \infty} \frac{\log N_{\varphi}(\varepsilon, d)}{d+\varepsilon^{-1}} \le 2 \limsup_{d+\varepsilon^{-1} \rightarrow \infty} \frac{\log \varphi(d)}{d+\varepsilon^{-1}}=0.$$ This implies weak tractability of the discrepancy with respect to $\| \cdot \|_{\psi_{\alpha, \varphi}}$. \hfill $\qed$

\subsection{The proof of Theorem~\ref{thm2}}\label{subsec:thm2}

First we show the corresponding result for $N_{\alpha}(\varepsilon,d)$ which is based on the norm $\|\cdot\|_{\alpha}$. Recall that for a measurable function $f: [0,1]^d \to \mathbb{R}$, we defined $\| f\|_\alpha= \sup_{p \ge 1} p^{-1/\alpha} \|f\|_{L_p}$. Let us start with a lower bound for the initial discrepancy. We have
\begin{eqnarray}\label{lbd:intitdiscalpha}
\| \Delta_\emptyset \|_{\alpha} & = & \sup_{p \ge 1} \frac{1}{ p^{1/\alpha}} \frac{1}{(p+1)^{d/p}}\nonumber\\ & \ge & \frac{1}{(d\log (d+1))^{1/\alpha}} \frac{1}{(1+d \log (d+1))^{1/\log (d+1)}}\nonumber\\
& \ge & \frac{1}{4 (d\log (d+1))^{1/\alpha}},
\end{eqnarray}
where we have chosen $p=d \log (d+1)$. The final estimate follows from the fact that $$d \mapsto \frac{1}{(1+d \log (d+1))^{1/\log (d+1)}}$$ attains its minimum in $d=20$ with minimal value $0.257944\ldots$.

Now let $d \in \NN$. Then from Gnewuch~\cite[Theorem~3]{gne05} we obtain that
$$
 \EE \, \|\Delta_{\cP}\|_{L_d} \le 2^{5/4} 3^{-3/4} \, N^{-1/2}
$$
and from Aistleitner and Hofer \cite[Corollary~1]{aisthof} that for any $q \in (0,1)$
$$
 \PP \left[\|\Delta_{\cP}\|_{L_\infty} \le 5.7\sqrt{4.9+\log((1-q)^{-1})} d^{1/2} N^{-1/2}\right] \ge q,
$$
where the expectation and probability are with respect to the point set $\cP$ consisting of independent and uniformly distributed points. Now Markov's inequality implies that there exists an $N$-element point set $\cP$ in $[0,1)^d$ such that
$$
 \|\Delta_{\cP}\|_{L_d} \le a \,  N^{-1/2} \qquad \text{and}\qquad
 \|\Delta_{\cP}\|_{L_\infty} \le a \, d^{1/2} N^{-1/2}
$$
provided that
$$
 1> \frac{2^{5/4}}{3^{3/4} a}+\exp\left(4.9-\left(\frac{a}{5.7}\right)^2\right).
$$
For this point set $\cP$, we obtain
\begin{eqnarray}\label{estalphadisc}
\|\Delta_{\cP}\|_{\alpha}  & = & \sup_{p \ge 1} p^{-1/\alpha} \|\Delta_{\cP}\|_{L_p}\nonumber\\
& = & \max\left\{\sup_{p \le d} p^{-1/\alpha} \|\Delta_{\cP}\|_{L_p},\sup_{p \ge d} p^{-1/\alpha} \|\Delta_{\cP}\|_{L_p}\right\}\nonumber\\
& \le & a\, N^{-1/2} \, \max\left\{\sup_{p \le d} p^{-1/\alpha} ,\sup_{p \ge d} p^{-1/\alpha} d^{1/2}\right\}\nonumber\\
& = & a \, N^{-1/2}\, \max\big\{1,d^{1/2-1/\alpha}\big\}.
\end{eqnarray}
Combining \eqref{lbd:intitdiscalpha} and \eqref{estalphadisc}, we obtain the upper bound
\begin{eqnarray*}\label{bdNst}
 N_{\alpha}(\varepsilon,d) & \le &  \min\left\{N \in \NN \,: \, \frac{a}{N^{1/2}} \max\big\{1,d^{1/2-1/\alpha}\big\} \le \varepsilon \, \frac{1}{4 (d\log (d+1))^{1/\alpha}}\right\}\nonumber\\
 & = & \left\lceil 16 \cdot a^2\, d^{2/\alpha}\, \max\big\{1,d^{1-2/\alpha}\big\}\, \big(\log(d+1)\big)^{2/\alpha} \varepsilon^{-2}\right\rceil\\
 & = & \left\lceil 16 \cdot a^2\, d^{\max\{1,2/\alpha\}}\,  \big(\log(d+1)\big)^{2/\alpha} \varepsilon^{-2}\right\rceil
\end{eqnarray*}
for all $\alpha\in[1,\infty)$. 
Note that we may choose $a=12.75$ leading to $16 a^2 =2601$.


Using the second part of Lemma~\ref{lem1}, we obtain
$$N_{\psi_{\alpha}}(\varepsilon,d) \le N_{\alpha}(\varepsilon',d)$$  
with $$\varepsilon' = \varepsilon\, \left(\frac{{\rm e}^{11/12}}{\sqrt{2 \pi}}\right)^{1/\alpha}\alpha^{-1/\alpha}\,. 
$$ From this we finally obtain the upper bound for $N_{\psi_{\alpha}}(\varepsilon,d)$. \hfill $\qed$

\bibliographystyle{plain}

\vspace{0.5cm}
\noindent{\bf Author's Addresses:}\\

\noindent Josef Dick, School of Mathematics and Statistics, The University of New South Wales, Sydney, NSW 2052, Australia.  Email: josef.dick@unsw.edu.au \\

\noindent Aicke Hinrichs, Institut f\"{u}r Analysis, Universit\"{a}t Linz, Altenbergerstra{\ss}e 69, A-4040 Linz, Austria. Email: aicke.hinrichs@jku.at\\

\noindent Friedrich Pillichshammer, Institut f\"{u}r Analysis, Universit\"{a}t Linz, Altenbergerstra{\ss}e 69, A-4040 Linz, Austria. Email: friedrich.pillichshammer@jku.at\\

\noindent Joscha Prochno, Institut f\"{u}r Mathematik und Wissenschaftliches Rechnen, Karl-Franzens-Universit\"{a}t Graz, Heinrichstra{\ss}e 36, A-8010 Graz, Austria. Email: joscha.prochno@uni-graz.at

\end{document}